\documentclass[10pt,reqno]{amsart}
\usepackage{amsmath}
\usepackage{amssymb}
\usepackage{amsthm}
\usepackage{enumitem}
\usepackage{graphicx,color,tikz}
\usepackage[pdftex]{hyperref}

\newlength{\defbaselineskip}
\newcommand{\setlinespacing}[1]%
           {\setlength{\baselineskip}{#1 \defbaselineskip}}

\numberwithin{equation}{section}

\newtheorem{thm}{Theorem}[]

\newtheorem{lem}{Lemma}
\newtheorem{prop}{Proposition}

\theoremstyle{definition}

\theoremstyle{remark}
\newtheorem{rem}{Remark}
\numberwithin{equation}{section}

\makeatletter
\@namedef{subjclassname@2020}{\textup{2020} Mathematics Subject Classification}
\makeatother

\begin{document}
\pagenumbering{arabic}\setcounter{page}{1}
\title[Quantum revivals and fractality]
{Quantum revivals and fractality for the Schr\"odinger equation}

\author{Gunwoo Cho, Jimyeong Kim, Seongyeon Kim, Yehyun Kwon and Ihyeok Seo}

\thanks{
	This work was supported by a KIAS Individual Grant (MG082902) at Korea Institute for Advanced Study and the POSCO Science Fellowship of POSCO TJ Park Foundation (S. Kim), NRF-RS-2023-00251435 (Y. Kwon), and NRF-2022R1A2C1011312 (I. Seo).}


\address{Department of Mathematics, Sungkyunkwan University, Suwon 16419, Republic of Korea}
\email{gwcho@skku.edu}

\email{jimkim@skku.edu}

\email{ihseo@skku.edu}

\address{School of Mathematics, Korea Institute for Advanced Study, Seoul 02455, Republic of Korea}
\email{synkim@kias.re.kr}

\address{Department of Mathematics, Changwon National University, Changwon 51140, Republic of Korea}
\email{yehyunkwon@changwon.ac.kr}

\begin{abstract}
We investigate the behavior of the Schr\"odinger equation under the influence of potentials, focusing on its relationship to quantum revivals and fractality. Our findings reveal that the solution displays fractal behavior at irrational times, while exhibiting regularity similar to the initial data at rational times. These extend the results of Oskolkov  \cite{O} and Rodnianski \cite{R2} on the free Schr\"odinger evolution to the general case regarding potentials.
\end{abstract}

\maketitle

\section{Introduction}\label{sec1}
In 1836, H. F. Talbot \cite{T1} observed an intriguing phenomenon in classical optics. When a plane wave passes through a periodic grating, the image of the diffracted wave is refocused and recovers the initial grating pattern with a specific periodicity. This phenomenon is now known as the Talbot effect. The distance between the grating and the point where the image is fully reconstructed is called the Talbot distance, denoted as $d_\mathrm{T}$. Lord Rayleigh \cite{R1} later derived a formula $d_\mathrm{T} = a^2/\lambda$, where $a$ is the spacing of the grating and $\lambda$ is the wavelength of the incident wave.

The Talbot effect, originally an optical experiment, can be described by using the Schr\"odinger evolution which leads us to mathematical studies of the physical consequences. One of such consequence is the phenomenon of quantum revivals, where a quantum wave packet is reconstructed at fractional multiples of a quantum period dependent on the Planck constant and energy levels \cite{BMS}.
In particular, Berry and Klein \cite{BK} used the free Schr\"odinger evolution to model the Talbot effect. They found a dichotomy between ``rational" and ``irrational" times, where a finite copy of the grating pattern reappears at every $t\in d_\mathrm{T} \mathbb Q$, but the wave function exhibits a fractal, nowhere-differentiable profile at every $t\notin d_\mathrm{T} \mathbb Q$.

As the grating is periodic in $x$, many authors have studied intensively the periodic Schr\"odinger equation 
\begin{equation}\label{e:schrodinger}
	\begin{cases}
		i\partial_t u+ \partial_{xx} u =0, \quad x\in \mathbb T:=\mathbb R/2\pi\mathbb Z, \ t\in \mathbb R, \\
		u(0,x)=f(x).
	\end{cases}
\end{equation}
Assuming the initial function $f$ is of bounded variation, Oskolkov \cite{O} proved that $u(t,x)$, the wave function, is continuous but nowhere differentiable in $x$ if $t\notin \pi \mathbb Q$, and necessarily contains (at most countable) discontinuities at every $t\in\pi \mathbb Q$ if $f$ has at least one discontinuity. Later, Kapitanski and Rodnianski \cite{KR} showed that $u(t,\cdot)$ has better regularity for $t\notin \pi \mathbb Q$. For every $t\in \pi \mathbb Q$, Taylor \cite{T3} found that if $f$ is the delta function, $u(t,\cdot)$ is a finite linear combination of translated delta functions with coefficients being Gauss sums.

The Talbot effect has also led to the study of the Schr\"odinger equation, particularly in relation to Berry's conjectures on the fractality of the graphs of $\mathrm{Re}\,u(t,\cdot)$, $\mathrm{Im}\,u(t,\cdot)$, and $|u(t,\cdot)|^2$. Berry \cite{B} calculated the fractal dimension of the temporal, spatial, and diagonal restrictions of the graph of the density function $|u(t,x)|^2$ for a step function $f$. Based on his calculations, he conjectured that for the $n$-dimensional linear Schr\"odinger equation confined in a box, the graphs of $\mathrm{Re}\,u(t,\cdot)$, $\mathrm{Im}\,u(t,\cdot)$, and $|u(t,\cdot)|^2$ have fractal dimension $D=n+1/2$ at almost all irrational times $t$.

Rodnianski  \cite{R2} rigorously justified Berry's conjecture in one dimension for the real and imaginary parts of the solution $e^{it\partial_{xx}}f$ with almost all irrational times, showing that the graphs have fractal dimension $D=3/2$ for every bounded variation $f$ not in $\bigcup_{\epsilon>0} H^{\frac{1}{2}+\epsilon}$. Recently, Chousionis, Erdo\u{g}an, and Tzirakis \cite{CET} extended this result to initial data in the larger class $BV\setminus\bigcup_{\epsilon >0}H^{s_0+\epsilon}$, $\frac{1}{2}\leq s_0<\frac{3}{4}$, and settled Berry's conjecture in one dimension for the density $|e^{it\partial_{xx}}f|^2$, showing that the dimension of the graph is $3/2$ at almost all $t\notin \pi\mathbb Q$ when $f$ is a step function with jumps only at rational points.

Drawing inspiration from previous work on the free Schr\"odinger evolution, we investigate quantum revivals and fractality for the Schr\"odinger equation in the presence of potentials:
\begin{equation}\label{PS}
	\begin{cases}
		iu_t + u_{xx} = V(x)u, \quad x\in\mathbb{T},\ t\in\mathbb{R}, \\
		u(0, x) = f(x).
	\end{cases}
\end{equation}
To state our results clearly, we introduce some notation. The space of functions of bounded variation on $\mathbb{T}$ is denoted by $BV$. For a $2\pi$-periodic function $f$, we define its Fourier coefficient as $\widehat f(k)=\frac1{2\pi}\int_r^{r+2\pi} e^{-ikx}f(x) dx$ for any $r\in\mathbb R$. For every $s\ge 0$, we denote the Sobolev space on $\mathbb{T}$ equipped with the norm 
\begin{equation}\label{e:snorm}
	\|f\|_{H^s} := \Big( \sum_{k\in\mathbb Z} \langle k\rangle^{2s} |\widehat f (k)|^2 \Big)^{1/2}
\end{equation}
 by $H^s$, where $\langle k\rangle=1+|k|$. We also use the notation $H^{s+}:=\bigcup_{\epsilon>0} H^{s+\epsilon}$ and $H^+:=H^{0+}$. Moreover, throughout this paper, if a Banach space of functions $\mathcal B^s$ is decreasing (in the sense of the set inclusion) with respect to a regularity index $s\in\mathbb R$, we use the following convenient notation:
\begin{equation}\label{e:cupcap}
	\mathcal B^{s+} := \bigcup_{\epsilon>0} \mathcal B^{s+\epsilon} , \quad \mathcal B^{s-} := \bigcap_{\epsilon>0} \mathcal B^{s-\epsilon}.
\end{equation}
Examples of $\mathcal B^{s}$ include the Sobolev space $H^s$, the Besov space $B_{p,q}^s$, and the H\"older space $C^s$.

Our first result concerns the dichotomy in regularity at rational and irrational times.

\begin{thm}\label{thm1}
Let $u$ be a solution to \eqref{PS} with $f\in BV$ and $V\in H^{+}$. If $f$ is continuous, then $u(t,x)$ is jointly continuous on $\mathbb{R}\times\mathbb{T}$. When $f$ is discontinuous, 
\begin{itemize}
\item
if $t\notin\pi\mathbb{Q}$, then $u(t,x)$ is continuous with respect to $x$,
\item if $t\in\pi\mathbb{Q}$, then $u(t,x)$ is bounded and necessarily contains (at most countable) discontinuities.
\end{itemize}
\end{thm}

We also determine the fractal dimension of the solution's graph at irrational time slices, in relation to the regularity of the potentials and initial data.

\begin{thm}\label{t:dim}
Let $u$ be a solution to \eqref{PS} with  $f\in BV$ and $V\in H^{+}$. Suppose that
\[	\sigma_0:= \sup\{\sigma\in \mathbb R\colon f\in H^\sigma\}<3/4	\]
(since $f\in BV$, we have $\sigma_0\ge 1/2$).
Let $r_0:= \sup\{r\in\mathbb R \colon V\in H^r\}$.
Then, for almost all $t\in \mathbb R\setminus \pi \mathbb Q$:
\begin{enumerate}
[leftmargin=0.6cm]
\item[$(i)$]  The upper Minkowski dimension of the graphs of $\mathrm{Re}\,u(t,\cdot)$, $\mathrm{Im}\,u(t,\cdot)$, and $|u(t,\cdot)|^2$ is less than or equal to $\max\big\{\frac{3}{2},\frac{5}{2}-\sigma_0 -r_0 \big\}$.
\item[$(ii)$] The upper Minkowski dimension of the graphs of $\mathrm{Re}\,u(t,\cdot)$ and $\mathrm{Im}\,u(t,\cdot)$ is greater than or equal to $\frac{5}{2}-2\sigma_0$, provided that $\sigma_0-\frac12<r_0$.
\end{enumerate}
\end{thm}

\begin{rem}
In the above theorems, the potential space $H^+$ contains several interesting functions. 
A classical result by Haslam-Jones \cite{HJ} states that the Fourier coefficients of the unbounded function\footnote{A special example of \eqref{e:unbddpot} is the Coulomb type potential $|x|^{-\nu}, 0<\nu<1$, which is of significant interest particularly in atomic and molecular physics, solid-state physics, and particle physics.}
\begin{equation}\label{e:unbddpot}
	V(x)=\frac{\phi(x)}{|x|^\nu (\log |\kappa/x|)^a}, \quad -\pi \le x< \pi,
\end{equation}
where $\phi\in BV$, $0<\nu<1$, $\kappa>\pi$ and $a\in\mathbb R$, satisfies 
\[	|\widehat V(0)|\lesssim 1 \quad \text{and} \quad \widehat{V}(k)= O\big( |k|^{\nu-1} (\log|k|)^{-a} \big), \ \ k\neq 0. 	\]
Therefore, for every $a\in \mathbb R$ and $\nu<\frac12$, the unbounded potential \eqref{e:unbddpot} belongs to $H^+$. Moreover, since $C^\alpha \hookrightarrow H^s$ for $0<s<\alpha\le 1$ (\cite[Theorem 1.13]{MS-book}), we conclude that the space $H^+$ includes all classes of H\"older continuous periodic functions.
\end{rem}

\subsubsection*{Organization} In Section \ref{sec2}, we present the crucial smoothing estimate (Proposition \ref{thm2}) for the Duhamel part of the initial value problem \eqref{PS}, as well as the relevant results on the free Schr\"odinger evolution that we require. We then prove Theorems \ref{thm1} and \ref{t:dim}.
The remaining sections are devoted to proving the smoothing estimate. In Section \ref{sec3}, we first establish bilinear estimates in the Bourgain space for the potential term $Vu$ in \eqref{PS}. Building on this, we establish the well-posedness of \eqref{PS} in Section \ref{sec4}, and subsequently prove the smoothing estimate.

\subsubsection*{Notation}
In our use of inequalities, we adopt the symbol $C$ to represent a positive constant that may vary in each occurrence. For $A,B>0$, we express $A\lesssim B$ when $A\le CB$ with some constant $C>0$. We also use the notation $A\approx B$ to signify that $A\lesssim B$ and $B\lesssim A$.

\section{Proofs of the theorems} \label{sec2}

In this section, we prove Theorems \ref{thm1} and \ref{t:dim}. 
Let us first rearrange the equation in \eqref{PS} as follows: 
\begin{equation*}
	iu_t + \big(\partial_{xx}-\widehat{V}(0) \big)u = R(V,u),
\end{equation*}
where $R(V,u)$ is the function on $\mathbb R\times \mathbb T$ defined by
\[	R(V,u)(t,x)= \big(V(x)-\widehat V(0)\big) u(t,x).	\]
Then, by Duhamel's formula, the solution to the initial value problem \eqref{PS} can be equivalently written as
\[	u(t,x) = e^{it (\partial_{xx}-\widehat{V}(0))}f(x) -i\int_0^t e^{i(t-t')(\partial_{xx}-\widehat{V}(0))}R(V,u)(t',x)dt'. \]

We shall prove that the Duhamel term 
\begin{equation}\label{e:inhom}
 	\mathcal P(t,x):= -i\int_0^t e^{i(t-t')(\partial_{xx}-\widehat{V}(0))}R(V,u)(t',x)dt'
\end{equation}
is in fact continuous on $\mathbb R\times \mathbb T$. For the purpose we obtain the following smoothing property for $\mathcal{P}(t,\cdot)$, which is the key ingredient in this paper. We provide the proof of the property in Section \ref{subsec3}. 
\begin{prop}\label{thm2}
Suppose that $V\in H^r$ and $f\in H^s$ for $r\geq 0$ and $0<s<r+1$.  If 
\[	a\le r \quad \text{and} \quad a<\min\{1+r-s, 1/2\},	\] 
then $\mathcal{P}(t,\cdot)\in H^{s+a}$ for every $t\in \mathbb R$, and is continuous in the time variable $t$. 
\end{prop}

For  the evolution $e^{it(\partial_{xx}-\widehat{V}(0) )}$, we make use of the following known result due to Oskolkov \cite[Proposition 14 and page 390]{O}:

\begin{prop}\label{Tal1}
Suppose that  $g\in BV$. 
\begin{enumerate}
[leftmargin=0.7cm]
	\item[$(i)$] If $t\notin\pi\mathbb{Q}$, then $e^{it\partial_{xx}}g$ is a continuous function of $x$. If $g$ has at least one discontinuity on $\mathbb T$ and $t\in\pi\mathbb{Q}$, then $e^{it\partial_{xx}}g$ necessarily contains discontinuities.
	\item[$(ii)$] If $g$  is continuous, then $e^{it\partial_{xx}}g$ is jointly continuous in temporal and spatial variables. 
\end{enumerate}
\end{prop}
\begin{rem}\label{rem1}
The quantization results of Berry and Klein \cite{BK} and Taylor \cite[(2.3)]{T3} state that if $t\in\pi\mathbb Q$ then $e^{it\partial_{xx}}g$ is a linear sum of finitely many translates of $g$ (see \cite[Theorem 2.14]{ET-book}). Hence, in Proposition \ref{Tal1} $(i)$, $e^{it\partial_{xx}}g\in BV$ whenever $t\in \pi \mathbb {Q}$, so the possible discontinuities in this case are at most countable.
\end{rem}

\begin{proof}[Proof of Theorem \ref{thm1}]
By the assumption $V\in  H^+$ it is possible to pick a sufficiently small number $\alpha\in(0,\frac18)$ such that $V\in H^{2\alpha}$. Since $f\in BV$  it follows that $f\in H^{\frac{1}{2}-}$, and in particular, $f\in H^{\frac12-\alpha}$. Thus, making use of Proposition \ref{thm2} (with $r=2\alpha$, $s=\frac12-\alpha$ and $a=2\alpha$), we conclude that $\mathcal P(t,x) \in C_tH_x^{\frac12 +\alpha}$. Moreover, from the Sobolev embedding
\begin{equation}\label{e:sob-emb}
	H^s \hookrightarrow C^{s-\frac{1}{2}} \quad \text{for} \quad s>1/2,
\end{equation} 
it follows that 
\begin{equation}\label{e:joint-conti}
	\mathcal P(t,x) \in C_tC_x^\alpha.
\end{equation}

For $t\notin \pi \mathbb Q$, Proposition \ref{Tal1} $(i)$ shows that 
\begin{equation}\label{e:evolution}
	e^{-it(\partial_{xx}-\widehat V(0))}f = e^{it \widehat V(0)}e^{-it\partial_{xx}}f 
\end{equation}
is continuous. Hence, combined with \eqref{e:joint-conti}, it follows from \eqref{e:inhom} that 
\begin{equation}\label{e:sol}
	u(t,\cdot)=e^{-it(\partial_{xx}-\widehat V(0))}f(\cdot)+\mathcal P(t,\cdot)
\end{equation}
is also continuous.

If $t\in \pi \mathbb Q$ and $f$ is discontinuous, then it follows from Proposition \ref{Tal1} $(i)$ and Remark \ref{rem1} that  the evolution \eqref{e:evolution} is of bounded variation and contains (at most countable) discontinuities. Thus, \eqref{e:joint-conti} tells us that the solution \eqref{e:sol} is a discontinuous bounded function with at most countable discontinuity. 

If $f$ is continuous, then the function \eqref{e:evolution} is also continuous by Proposition \ref{Tal1} $(ii)$. Hence, combining with \eqref{e:joint-conti} we see that the solution \eqref{e:sol} is jointly continuous on $\mathbb R\times \mathbb T$.
\end{proof}

Now we prove Theorem \ref{t:dim}.  Let us first recall the inhomogeneous Besov space $B^s_{p,q}$, where 
$1\le p,q\le\infty$ and $s\ge 0$, on the periodic domain $\mathbb T$
and its properties that we need.
It is well-known\footnote{See, for example, \cite{ST87}  (Remark 4 on p. 164 and Theorem (i), (v) on pp. 168--169).} that $B_{2,2}^s=H^s$ for every $s$ and $B^\alpha_{\infty, \infty} = C^\alpha$ for $\alpha\in (0,\infty)\setminus \mathbb N$.  By complex interpolation between $B_{1,\infty}^{s_1}$ and 
$B_{\infty,\infty}^{s_2}$ (see \cite[Theorem 6.4.5]{BL1}) and H\"older's inequality, we have, for $s_1\neq s_2$,
\begin{equation}\label{e:inter}
	B_{1,\infty}^{s_1}\cap B_{\infty,\infty}^{s_2} \subset B_{2,\infty}^{\frac{s_1+s_2}{2}} \hookrightarrow B_{2,2}^{\frac{s_1+s_2}2 -} = H^{\frac{s_1+s_2}{2} -}.	
\end{equation}

We also make use of the following lemmas due to Chousionis--Erdo\u{g}an--Tzirakis \cite{CET} and Deliu--Jawerth \cite{DJ}. 
\begin{lem}[\cite{CET}] \label{Linear}
Let $\frac12\le \sigma_0< \frac{3}{4}$ and suppose $g\in BV\setminus H^{\sigma_0+}$. Then, for almost all $t\in \mathbb R\setminus \pi\mathbb Q$, we have $e^{it\partial_{xx}}g \in C^{\frac{1}{2}-} \setminus B_{1,\infty}^{2\sigma_0-\frac{1}{2}+}$.
\end{lem}

\begin{lem}[\cite{DJ}]\label{FDL}
Let $0<s<1$. Assume that $f\colon\mathbb{T}\to \mathbb{R}$ is continuous and $f\notin B_{1,\infty}^{s+}$. Then the upper Minkowski dimension of the graph of $f$ is at least  $2-s$.
\end{lem}

\begin{proof}[Proof of Theorem \ref{t:dim}]
Let us first prove the part $(i)$. By the definition of $\sigma_0$ and $r_0$ it is clear that 
\[	f\in H^{\sigma_0-} \quad \text{and} \quad V\in H^{r_0-}.	\]
By applying Proposition \ref{thm2} (with $s=\sigma_0-\epsilon$ and $r=r_0-2\epsilon$ for an infinitesimal $0<\epsilon\ll 1$) we have 
\begin{equation}\label{e:p-sob}
	\mathcal P(t,\cdot)\in H^{\sigma_0+a}, \quad \forall t\in \mathbb R,
\end{equation}
whenever 
\[	a<\min\{r_0, 1+r_0-\sigma_0, 1/2\} =\min\{r_0, 1/2\}=:a_0.	\]
Hence, by the Sobolev embedding \eqref{e:sob-emb} we have
\begin{equation}\label{e:P-hol}
	\mathcal P(t,\cdot)\in C^{\sigma_0+a-\frac12} , \quad \forall t\in\mathbb R,
\end{equation}
for every $a<a_0$. On the other hand, by Lemma \ref{Linear}, we have
\begin{equation}\label{e:evol-hol}
	e^{it\partial_{xx}}f \in C^{\frac12-} \quad \text{for almost all} \ \  t\in \mathbb R\setminus \pi \mathbb Q.
\end{equation}
Hence, it follows from \eqref{e:sol}, \eqref{e:P-hol} and \eqref{e:evol-hol} that
\begin{equation}\label{e:u-hol}
	u(t,\cdot) \in C^{\min\{\frac12, \sigma_0+a-\frac12\}-}  \quad \text{for almost all} \ \  t\in \mathbb R\setminus \pi \mathbb Q
\end{equation}
for every $a<a_0$. It is obvious that the same statement is still valid for $\mathrm{Re}\,u(t,\cdot)$, $\mathrm{Im}\,u(t,\cdot)$ and $|u(t,\cdot)|^2$ in place of $u(t,\cdot)$ in \eqref{e:u-hol}. We now use the following classical result on the upper Minkowski dimension for H\"older continuous functions (for a proof we refer the reader to \cite[Corollary 11.2]{F}):
\begin{lem}\label{l:upper}
Let $0\le \alpha\le 1$. If a function $f\colon \mathbb T\to \mathbb R$ is $C^\alpha$, then the upper Minkowski dimension of the graph of $f$ is at most $2-\alpha$.  
\end{lem}

Indeed, by Lemma \ref{l:upper}, the upper Minkowski dimension of the graphs of $\mathrm{Re}\,u(t,\cdot)$, $\mathrm{Im}\,u(t,\cdot)$ and $|u(t,\cdot)|^2$ is at most  
\[	2-\min\{1/2, \sigma_0+a_0-1/2\}= \max\{3/2, 5/2 -\sigma_0-a_0\}	\]
for almost all $t\in \mathbb R\setminus \pi \mathbb Q$. We notice further that if $a_0=1/2$, then $5/2 -\sigma_0-a_0=2-\sigma_0\in (5/4,3/2]$, so the maximum is equal to $3/2$. Therefore, we obtain $(i)$.

Let us now prove the part $(ii)$. Since $f\notin H^{\sigma_0+}$, we see that for almost every $t$, neither $\mathrm{Re\,} e^{it\partial_{xx}}f$ nor $\mathrm{Im\,} e^{it\partial_{xx}}f$ belong to $H^{\sigma_0 +}$ (see \cite[Lemma 3.2]{CET}). It follows from \eqref{e:evol-hol} and the embedding $B_{1,\infty}^{s_1}\cap B_{\infty,\infty}^{s_2} \subset H^{\frac{s_1+s_2}{2} -}$ (see \eqref{e:inter}) that, for almost all $t\in \mathbb R\setminus 2\pi\mathbb Q$, both the real and the imaginary parts of $e^{it\partial_{xx}}f$ do not belong to $B_{1,\infty}^{2\sigma_0-\frac12+}$.  On the other hand, by \eqref{e:p-sob} we see $\mathcal P(t,\cdot)\in B_{1,\infty}^{\sigma_0+\min\{r_0,\frac12 \}-}$ for all $t\in \mathbb R$. Combining these we have
\[	u(t,\cdot)= \underbrace{\mathrm{Re}\,(e^{it(\partial_{xx}-\widehat V(0))}f)}_{\notin B^{2\sigma_0-\frac12+}_{1,\infty}\text{ for a.e. }t} 
+i\underbrace{\mathrm{Im}\,(e^{it(\partial_{xx}-\widehat V(0))}f)}_{\notin B^{2\sigma_0-\frac12+}_{1,\infty}\text{ for a.e. }t}
+\underbrace{\mathcal P(t,\cdot)}_{\in B^{\sigma_0+\min\{r_0,\frac12\}- }_{1,\infty} \text{ for all }t}. 	\]
From this we conclude that if $2\sigma_0-\frac12<\sigma_0+\min\{r_0,\frac12\}$, that is, either $r_0\ge \frac12$ or $\sigma_0-\frac12 <r_0<\frac12$,  then for almost all $t\notin 2\pi \mathbb Q$ neither the real nor the imaginary parts of $u(t,\cdot)$ belong to $B^{2\sigma_0-\frac12+}_{1,\infty}$. By Lemma \ref{FDL}, we conclude that the graphs of real and imaginary parts of $u(t,\cdot)$ have Minkowski dimension $\ge 2-(2\sigma_0-\frac12)=\frac 52-2\sigma_0$.
\end{proof}

\section{Bilinear estimate in the Bourgain space} \label{sec3}
In this section, we obtain bilinear estimates for $R(V,u)$  in the Bourgain space, which are essential in proving Proposition \ref{thm2}. 
\subsection{The Bourgain space} 
For $s,b\in\mathbb{R}$, we denote by  $X^{s,b}$ the closure of the set of Schwartz functions $\mathcal S(\mathbb R)_t C^\infty (\mathbb T)_x$ under the norm
\[	\|u\|_{X^{s,b}} := \| \langle k\rangle^s\langle\tau+k^2\rangle^b\widetilde{u}(\tau,k)\|_{L^2_{\tau}l^2_{k}(\mathbb R\times \mathbb Z)},	\]
where $\widetilde{u}$ denotes the space-time Fourier transform of $u$ defined by
\[	\widetilde{u}(\tau,k)=\int_{\mathbb{R}} e^{-i\tau t}\widehat{u}(t,k) dt
	=\frac{1}{2\pi} \int_{\mathbb{R}}\int^{2\pi}_{0} e^{-i(\tau t+kx)} u(t,x) dx dt, \quad (\tau, k)\in \mathbb R\times \mathbb Z.	\]
The $X^{s,b}$-space is also called the \emph{Bourgain space} or \emph{dispersive Sobolev space}. We also define, for a closed interval $I\subset \mathbb{R}$, the restricted space $X^{s,b}_I$ as the Banach space of functions on $I\times \mathbb T$ equipped with the norm
\[	\|u\|_{X^{s,b}_I} = \inf\big\{\|w\|_{X^{s,b}} \colon w\vert_{I\times\mathbb T}=u \big\}.	\]
We notice that $X^{s,b}_\mathbb R=X^{s,b}$.

The followings are some basic properties of the Bourgain space that we need to prove the smoothing estimate (Proposition \ref{thm2}). The proofs of the properties, which we omit, can be obtained by routine adaptation of those of \cite[Corollary 2.10, Lemma 2.11 and Proposition 2.12]{T2} in the spatially periodic setting and time translation argument.  
\begin{lem}\label{bourgain}
Let $s\in\mathbb{R}$, $b>\frac12$ and $I\subset \mathbb R$ a closed interval. Then $X_I^{s,b} \hookrightarrow C(I; H^s)$ and
\[	\sup_{t\in I}\|u(t,\cdot)\|_{H^s}\leq C \|u\|_{X_I^{s,b}},	\]
where $C$ is a constant depending only on $b$.
\end{lem}	

\begin{lem}\label{lem2}
Let $s\in\mathbb{R}$, $-\frac12<b<b^\prime<\frac12$ and $I$ a closed interval of length $\delta$. Then
\[	\|u\|_{X^{s,b}_I} \leq C\delta^{b^\prime-b}\|u\|_{X^{s,b^\prime}_I},	\]
where the constant $C$ depends only on $b$ and $b'$.
\end{lem}

\begin{lem}\label{lem1}
Let $s\in\mathbb{R}$, $\frac12<b\leq 1$ and $I=[t_0, t_0+\delta]$ for $t_0\in\mathbb R$ and $0<\delta\le 1$. Then, for $t\in I$, we have
\begin{gather*}
	\| e^{i(t-t_0)\partial_{xx}}f\|_{X^{s,b}_I} \leq C\|f\|_{H^{s}}, \\
	\Big \|\int^t_{t_0} e^{i(t-t')\partial_{xx}}F(t',\cdot)dt' \Big\|_{X^{s,b}_I} \leq C\| F\|_{X^{s,b-1}_I},
\end{gather*}
where  $C$ depends only on $b$. 
\end{lem}

\subsection{Bilinear estimate for $R(V,u)$} \label{subsec1} 
In this section we estimate $R(V,u)=(V-\widehat V(0)) u$ in the $X^{s,b}$-space.

\begin{prop}\label{prop1}
Let $r\ge 0$ and $0< s<1+r$, and suppose that  
\[	a\le r, \quad a<1+r-s \quad \text{and} \quad a<1/2.	\]
Then, for every interval $I$,
\begin{equation}\label{e:bil-est}
	\|R(V,u)\|_{X_I^{s+a,b'-1}}\lesssim \|V\|_{H^r}\|u\|_{X_I^{s,b}}
\end{equation}
provided that $b, b' \in (\frac12, \frac12+\epsilon)$ for an $\epsilon>0$ small enough.
\end{prop}

We first recall  from \cite[Lemma 3.3]{ET2} the following simple lemma which is used several times in proving the proposition. 
\begin{lem}\label{lem3}
Let us define, for $\beta\ge 0$,
\begin{equation}\label{e:phi-k}
	\phi_{\beta}(k)
	:= \sum_{|n|\leq|k|}\frac{1}{\langle n\rangle^{\beta}}  
	\approx
	\begin{cases}
		1,&\beta>1,\\
		\log(1+\langle k\rangle),&\beta=1,\\
		\langle k\rangle^{1-\beta},&\beta<1.
	\end{cases}
\end{equation}
If $\beta \geq \gamma \geq 0$ and $\beta +\gamma>1$, then
\begin{equation*}
\sum_n \frac{1}{\langle n-k_1\rangle^{\beta}\langle n-k_2\rangle^{\gamma}} 
	\lesssim \frac{\phi_{\beta}(k_1-k_2)}{\langle k_1-k_2\rangle^{\gamma}}.
\end{equation*}
\end{lem}

\begin{proof}[Proof of Proposition \ref{prop1}]
First, let us prove \eqref{e:bil-est} by replacing the restricted spaces $X_I^{s+a,b'-1}$ and $X_I^{s,b}$  with $X^{s+a,b'-1}$ and $X^{s,b}$, respectively.  We write  
\begin{align}
	&\|R(V,u)\|_{X^{s+a,b'-1}}^2 \nonumber \\
	&= \int_{\mathbb R} \sum_{k\in \mathbb Z} \bigg | \sum_{l\in\mathbb Z\setminus\{k\} } \langle k\rangle^{s+a} \langle \tau +k^2 \rangle^{b'-1} \widehat{V}(k-l) \widetilde{u}(\tau,l)\bigg |^2 d\tau	\nonumber \\
	&= \int_{\mathbb R} \sum_{k\in \mathbb Z} \bigg | \sum_{l\in \mathbb Z\setminus \{k\}} M(k,l,\tau)
		 \langle k-l\rangle^{r} \widehat{V}(k-l) \langle l \rangle^{s}\langle \tau + l^2\rangle^{b} \,\widetilde{u}(\tau, l) \bigg |^2 d\tau ,  \label{e:RVu-X}
\end{align}
where 
\[	M(k,l,\tau):=\langle k\rangle^{s+a} \langle \tau +k^2 \rangle^{b'-1} \langle l \rangle^{-s} \langle k- l \rangle^{-r} \langle \tau + l^2 \rangle^{-b}.	\]
By the Cauchy--Schwarz and H\"older inequalities \eqref{e:RVu-X} is bounded by
\[	\left(\sup_{\tau,k}\sum_{l\in \mathbb Z\setminus \{k\}}M(k,l,\tau)^2 \right)
	\int \sum_{k\in \mathbb Z}\sum_{l\in \mathbb Z\setminus\{k\}}\langle k-l \rangle^{2r} |\widehat{V}(k-l)|^2 \langle l\rangle^{2s}\langle \tau + l^2\rangle^{2b}|\widetilde{u}(\tau,l)|^2 d\tau. 	\]
In the $\tau$-integration we replace the summation $\sum_{l\in \mathbb Z\setminus\{k\}}$ with $\sum_{l\in \mathbb Z}$, and then apply Young's inequality for the convolution, i.e., $\sum_{k\in\mathbb Z}|\sum_{l\in\mathbb Z}a_{k-l}b_l | \le \|a\|_{\ell^1(\mathbb Z)}\|b\|_{\ell^1(\mathbb Z)}$. Consequently we obtain the following estimate
\[	\|R(V,u)\|_{X^{s+a,b'-1}}^2\le \left(\sup_{\tau,k}\sum_{l\in \mathbb Z\setminus \{k\}}M(k,l,\tau)^2\right) \|V\|_{H^r}^2 \|u\|_{X^{s,b}}^2.	\]
Hence it remains to prove that 
\[	\sup_{\tau,k}\sum_{l\in \mathbb Z\setminus\{k\}}M(k,l,\tau)^2 <\infty.	\]

From the assumption, it is clear that $2b\ge 2-2b'$. Hence, by the easy inequality $\langle \tau +m\rangle \langle \tau +n\rangle \ge \frac12 \langle n-m\rangle$, we have
\begin{align}
	\sum_{l\in \mathbb Z\setminus\{k\}} M(k,l,\tau)^2
	&\le \sum_{l\in \mathbb Z\setminus\{k\}}\frac{\langle k\rangle^{2s+2a} \langle l\rangle^{-2s} \langle k-l \rangle^{-2r} }{\langle \tau +k^2 \rangle^{2-2b'} \langle \tau + l^2\rangle^{2-2b'}} \nonumber	\\
	&\lesssim \sum_{l\in \mathbb Z\setminus\{k\}} \frac{\langle k\rangle^{2s+2a} \langle l\rangle^{-2s} \langle k-l \rangle^{-2r} }{\langle l^2 - k^2 \rangle^{2-2b'}} \nonumber	\\
	&\lesssim  \sum_{l\in \mathbb Z\setminus\{\pm k\}} \frac{\langle k\rangle^{2s+2a} \langle l\rangle^{-2s} \langle k- l \rangle^{-2r} }{\langle l+ k \rangle^{2-2b'} \langle l- k \rangle^{2-2b'}} + \langle k \rangle^{2a-2r}. \label{e:summation}
\end{align}
Since $a\leq r$ it is obvious that $\langle k \rangle^{2a-2r}\le 1$, so we need only to show the  summation in \eqref{e:summation} is bounded uniformly in $k\in \mathbb Z$. That is, we have to prove that for some $\epsilon>0$
$$S:=\sum_{l\in\mathbb Z\setminus \{\pm k\}} \langle l\rangle^{-2s}\langle k-l\rangle^{-2r+\epsilon-1}\langle k+l\rangle^{\epsilon-1}\lesssim \langle k \rangle^{-2(s+a)}.$$
Here, for the sake of brevity we renamed $\epsilon$ as $\epsilon/2$, so $2b,2b'\in(1,1+\epsilon)$.
Given $k\neq0,$ We also write simply $\phi_\beta$ to mean $\phi_\beta(ck)$ with $c$ a constant because they have the same order of magnitude.

We can assume $k>0$ by the symmetry and because for $k=0$ the series converges for any $\epsilon<1/2$.
By using Lemma \ref{lem3}, we have 
\begin{equation}
S \lesssim 
\begin{cases} \langle k\rangle^{-2r-1+\epsilon-\min\{2s,1-\epsilon\}} \phi_{\max\{2s,1-\epsilon\}} \quad \textrm{if} \quad |l|\leq k/2, \\
\langle k\rangle^{-2s+\epsilon-1} \phi_{2r+1-\epsilon} + \langle k\rangle^{-2s+\epsilon-1-2r} \phi_{1-\epsilon} \quad \textrm{if} \quad k/2\leq |l|\leq 3k/2, \\
\langle k\rangle^{-2s-2r-1+2\epsilon} \quad \textrm{if} \quad 3k/2\leq |l|.
\end{cases}
\end{equation}
In the second case, the second term is negligible because $\langle k\rangle^{-2r}\phi_{1-\epsilon} \lesssim \phi_{2r+1-\epsilon}$ is trivial for $r=0$ and $0<\epsilon<2r$. In both situations, $\phi_{2r+1-\epsilon}\lesssim \langle k\rangle^{\epsilon}$ and we deduce $S\lesssim \langle k\rangle^{2\epsilon-1-2s}$.
This bound also holds in the third case (because $r\ge0$).
Hence $a<1/2$ assures $S\lesssim \langle k\rangle^{-2(s+a)}$ when $\epsilon$ is small enough.
It only remains to check the first case.
Since $\phi_{\max\{2s,1-\epsilon\}}\lesssim \langle k\rangle^{\epsilon}$ for $s\ge 1/2$, we deduce $S \lesssim \langle k\rangle^{-2r-2+3\epsilon}$ which assure $S\lesssim \langle k\rangle^{-2(s+a)}$ if $a<1+r-s$ and $\epsilon$ is small enough.
For $0<s<\frac12$, we have $\langle k\rangle^{-\min\{2s,1-\epsilon\}} \phi_{\max\{2s,1-\epsilon\}}\lesssim \langle k\rangle^{-2s+\epsilon}$ and $S\lesssim \langle k\rangle^{-2(s+a)}$ if $a<1/2$ and $\epsilon$ is small enough.

We now prove \eqref{e:bil-est}. Let $\chi_I\in C_0^\infty(\mathbb R)$ be such that $\chi_I=1$ on $I$. By the definition of the restricted Bourgain space and the estimate what we have obtained, we see that
\[	\|R(V,u)\|_{X_I^{s+a,b'-1}} \le \|\chi_I R(V,u)\|_{X^{s+a,b'-1}} \lesssim \|V\|_{H^r} \|\chi_I u\|_{X^{s,b}}.	\]
Taking infimum over all functions $w$ such that $w\vert_{ I\times \mathbb R}=u$ on the right side, we get the estimate \eqref{e:bil-est}.
\end{proof}

\section{Local well-posedness and smoothing estimate} \label{sec4}
In this section, we make use of Proposition \ref{prop1} and employ the contraction mapping principle to establish local well-posedness in the Sobolev spaces for the initial value problem \eqref{PS}. We then prove Proposition \ref{thm2}.

Let $u$ be the solution to the equation \eqref{PS} and set $v(t,x):=e^{it\widehat V(0)}u(t,x)$. Since 
\begin{align*}
	iv_t+v_{xx} &= e^{it\widehat V(0)}\big(i u_t - \widehat V(0) u + u_{xx} \big) \\
	&= e^{it\widehat V(0)} \big(V -\widehat V(0) \big)u =\big(V -\widehat V(0) \big)v	= R(V,v),
\end{align*}
the initial value problem \eqref{PS} is equivalent to the following:
\begin{equation}\label{MPS}
	\begin{cases}
	iv_t + v_{xx} =R(V,v),\\
	v(0, x) = f(x).	
	\end{cases}
\end{equation}
Let us recall the definition of $\mathcal P(t,x)$ from \eqref{e:inhom} and note that
\begin{equation} \label{cov}
	v(t,x)-e^{it\partial_{xx}}f = e^{it\widehat V(0)} \big( u(t,x) - e^{it(\partial_{xx}-\widehat V(0))} f \big) =e^{it\widehat{V}(0)}\mathcal{P}(t,x).
\end{equation}
Since $\|\mathcal{P}(t,x)\|_{H^s} = \|e^{it\widehat{V}(0)}\mathcal{P}(t,x)\|_{H^s}$, in order to prove Proposition \ref{thm2} it is enough to prove the smoothing estimate for $e^{it\widehat{V}(0)}\mathcal{P}(t,x)$ instead of $\mathcal P(t,x)$. Therefore, in this section, we may and shall consider the equation \eqref{MPS} instead of \eqref{PS}.

\subsection{Local well-posedness}\label{subsec2}
We now prove the initial value problem \eqref{MPS} is locally well-posed in $H^s$. We use the notation $X^{s,b}_\delta :=X^{s,b}_{[0,\delta]}$.
\begin{thm}\label{local}
	Let $V\in H^r$ for $r\geq 0$ and suppose $0< s < 1+r$ and $\frac12 <b<\frac 12+\epsilon$ for some $\epsilon>0$ small enough. 
	For every  $g\in H^s$, there exists a time $\delta>0$ and a unique solution $v_g\in X_{\delta}^{s,b}$ for the integral equation
	\begin{equation}\label{e:int-eq}
		v_g(t,x) = e^{it\partial_{xx}}g-i\int_0^t e^{i(t-t')\partial_{xx}}R(V,v_g)(t') dt'.	
	\end{equation}
	Furthermore, the mapping $g\mapsto v_g$ is Lipschitz continuous  with the estimate
	\begin{equation}\label{prop2}
		\|v_g\|_{X_\delta^{s,b}}\lesssim \|g\|_{H^{s}}.
	\end{equation}
\end{thm}
\begin{rem}
By Lemma \ref{bourgain} we have $X_{\delta}^{s,b} \hookrightarrow C([0,\delta], H^s)$. Thus, the theorem establishes the local well-posedness of \eqref{MPS} (hence \eqref{e:schrodinger}) in $H^s$,  in the sense of \cite[Definition 3.4]{T2}. 
\end{rem}
\begin{rem}\label{r:unif-t}
In fact, as can be seen in the proof (see \eqref{e:delta} below), the small time $\delta$ is  independent of the initial data $f\in H^s$. Since the initial value problem \eqref{MPS} is invariant under time translations, we can patch the solution \eqref{e:int-eq} along $t\in \mathbb R$ by repeatedly applying Theorem \ref{local} to obtain the global solution $v\in C(\mathbb R, H^s)$ to \eqref{MPS}. Also, the estimate \eqref{prop2} combined with Lemma \ref{bourgain} implies the following global bound
\begin{equation}\label{remark}
	\|v(t)\|_{H^s}\lesssim e^{|t|}\|f\|_{H^s}, \quad \forall t\in \mathbb R.
\end{equation}
Indeed, by using Lemma \ref{bourgain} and \eqref{prop2} in turn, we have $\sup_{t\in[0,\delta]}\|v\|_{H^s}\lesssim \|v\|_{X_\delta^{s,b}}\leq C\|f\|_{H^s}$,
and in the same way 
$\sup_{t\in[\delta,2\delta]}\|v\|_{H^s}\leq C\|v(\delta)\|_{H^s}\leq C^2\|f\|_{H^s}$.
Repeating this process, we get
$$\sup_{(j-1)\delta \leq t \leq j\delta}\|v(t)\|_{H^s} \leq C^{j} \|f\|_{H^s}\lesssim_\delta C^{t/\delta}\|f\|_{H^s}\lesssim_\delta e^{t}\|f\|_{H^s}$$
for all integers $j\ge1$. The case where $j\le0$ follows from a similar argument.  
\end{rem}

\begin{proof}[Proof of Theorem \ref{local}]
	For each $g\in H^s$ we define 
	\[	\Gamma_g (v)(t,x) := e^{it\partial_{xx}}g - i\int_0^t e^{i(t-t')\partial_{xx}}R(V,v)(t') dt'.	\]
	We aim to show that, for some small time $0<\delta<1$ and $K>0$ to be chosen later, the mapping $\Gamma_g$ is a contraction on the set
	\[	X:=\big\{w\in X_{\delta}^{s,b} \colon \|w\|_{X_{\delta}^{s,b}}\leq K\|g\|_{H^s}\big\}.	\]

	Let us first show that the mapping $\Gamma_g\colon X\to X$ is well-defined.  Lemmas \ref{lem2} and \ref{lem1}  imply that if $w\in X$, then
	\begin{align*}
		\|\Gamma_g (w)\|_{X_{\delta}^{s,b}}
		&\leq \|e^{it\partial_{xx}}g\|_{X_{\delta}^{s,b}}+\Big\|\int_0^t e^{i(t-t')\partial_{xx}}R(V,w)(t')dt'\Big\|_{X_{\delta}^{s,b}}\\
		&\leq C\|g\|_{H^s}+C \|R(V,w) \|_{X_{\delta}^{s,b-1}}\\
		&\leq C\|g\|_{H^s}+C\delta^{b'-b} \|R(V,w) \|_{X_{\delta}^{s,b'-1}}
	\end{align*}
	whenever $0<\delta< 1$ and  $\frac12<b<b'<1$. We then apply Proposition \ref{prop1} (with $a=0$)  to see that
	\[	\|\Gamma_g (w)\|_{X_{\delta}^{s,b}}
	\leq C\|g\|_{H^s} +C\delta^{b'-b}\|V\|_{H^r} \|w\|_{X_{\delta}^{s,b}},	\]
	provided that $\frac12<b<b'<\frac 12+\epsilon$ for $\epsilon>0$ small enough. Since $w\in X$ we get
	\[	\|\Gamma_g (w)\|_{X_{\delta}^{s,b}} \leq C_{0} \big(1+K\delta^{b'-b}\|V\|_{H^r} \big) \|g\|_{H^s}.	\]
	If we set $K \ge 2C_0$ and take $0<\delta<1$ so small that 
	\begin{equation}\label{e:delta}
		\delta^{b'-b}<1/(1+K\|V\|_{H^r}),
	\end{equation}
	then it follows that 
	\begin{equation*}
		\|\Gamma_g (w)\|_{X_{\delta}^{s,b}}\leq K\|g\|_{H^s}.
	\end{equation*} 
	Therefore, $\Gamma_g (X)\subset X$ for $g\in H^s$.
	
	Secondly, let us prove that the map $\Gamma_g\colon X\to X$ is a contraction. In a similar manner, by Lemmas \ref{lem2} and \ref{lem1} and Proposition \ref{prop1} we have 
	\begin{align*}
		\|\Gamma_g (w_1)-\Gamma_g (w_2)\|_{X_{\delta}^{s,b}}&\leq C\delta^{b'-b} \|R(V,w_1-w_2) \|_{X_{\delta}^{s,b'-1}}\\
		&\leq C\delta^{b'-b}\|V\|_{H^r} \|w_1-w_2\|_{X_{\delta}^{s,b}},
	\end{align*}
	which implies $\Gamma_g$ is a contraction on $X$ for $\delta$ satisfies \eqref{e:delta}, if we set $K\ge 2C$ here. By applying the contraction mapping principle, it follows that there exists a unique $v_g\in X$ solving the equation $\Gamma_g(v_g)=v_g$. 
	
	The continuity of the map $g\mapsto v_g$ also follows from utilizing Lemmas \ref{lem2} and \ref{lem1} and Proposition \ref{prop1} as in the above. Indeed, by \eqref{e:int-eq} we see that for $g, h\in H^s$,
	\begin{align*}
		&\|v_g-v_h\|_{X_\delta^{s,b}} \\
		&\le \|e^{it\partial_{xx}}(g-h)\|_{X_\delta^{s,b}} +\Big \| \int_0^t e^{i(t-t')\partial_{xx}} \big(R(V,v_g)(t')-R(V,v_h)(t')\big) dt' \Big\|_{X_\delta^{s,b}} \\
		&\le C\|g-h\|_{H^s} +C\delta^{b'-b}\|V\|_{H^r} \|v_g-v_h\|_{X_\delta^{s,b}} \\
		&\le C\|g-h\|_{H^s} +\frac{1}{2}\|v_g-v_h\|_{X_{\delta}^{s,b}},
	\end{align*}
	from which it follows that $\|v_g-v_h\|_{X_\delta^{s,b}}\lesssim \|g-h\|_{H^s}$.
\end{proof}

\subsection{Smoothing estimate: Proof of Proposition \ref{thm2}} \label{subsec3}
In this final section, we prove Proposition \ref{thm2}. 

Let $f$ and $V$ be given as in the statement of Proposition \ref{thm2}. 
It is enough to prove that for every $t_0 \ge 0$ there exists a constant $C>0$ such that 
\begin{equation}\label{want}
	\| \mathcal{P}(t_0,x)\|_{H^{s+a}} \leq C \|V\|_{H^r} \|f\|_{H^s}.
\end{equation} 
This gives regularity gain for $\mathcal P(t,\cdot)$ compared to the global estimate \eqref{remark} for the solution $v$.  In order to prove \eqref{want}, the bilinear estimate (Proposition \ref{prop1}) as well as the local well-posedness (Theorem \ref{local}) is crucial.

Let us invoke the $\delta>0$ in Theorem \ref{local} and pick  $m \in \mathbb N$ such that $(m-1)\delta\leq t_0 < m\delta$. We also set 
\begin{equation}\label{e:initial}
	v_j(x) = v(j\delta, x)   \quad \text{and} \quad I_j =[\delta j, \delta (j+1)].
\end{equation}
Applying Theorem \ref{local} with the initial data $v_j$, it is possible to write the solution $v$ as
\[	v(t) =e^{i(t-\delta j) \partial_{xx}} v_j  -i\int_{\delta j}^t e^{i(t-t')\partial_{xx}} R(V,v)(t') dt', \quad t\in I_j. 	\] 
By Lemmas \ref{bourgain} and \ref{lem1} we have, for $t\in I_j$,
\begin{align*}
\| v(t) - e^{i(t-\delta j)\partial_{xx}} v_j \|_{H^{s+a}}
	&\lesssim \Big\|\int_{\delta j}^t e^{i(t-t')\partial_{xx}}R(V,v)(t')dt' \Big\|_{X_{I_j}^{s+a,b}}\\
	&\lesssim \|R(V,v) \|_{X_{I_j}^{s+a,b-1}}
\end{align*}
whenever $\frac12<b\le 1$. Let us choose $b$ sufficiently close to $\frac12$, and then apply Proposition \ref{prop1} (with $b'=b$)  to see that
\[	\|R(V,v) \|_{X_{I_j}^{s+a,b-1}} \lesssim \|V\|_{H^r}\|v\|_{X_{I_j}^{s,b}}.	\]
Hence, by the estimate \eqref{remark}, we conclude that 
\begin{equation}\label{e:est-j}
	\| v(t) - e^{i(t-\delta j)\partial_{xx}} v_j \|_{H^{s+a}} \le C e^{t_0} \|V\|_{H^r} \|f\|_{H^s}, \quad t\in I_j,	
\end{equation}
for $j\in \{0,1,2,\cdots, m-1\}$.  We recall \eqref{cov} and write 
\begin{align*}
	&e^{it_0 \widehat V(0)}\mathcal P(t_0,x) 
	= v(t_0) -e^{it_0 \partial_{xx}} f \\
	&= v(t_0)-e^{i(t_0 -\delta (m-1))\partial_{xx}}v_{m-1}
		+\sum_{j=1}^{m-1} e^{i(t_0 -\delta j )\partial_{xx}} \big( v_{j}-e^{i\delta\partial_{xx}}v_{j-1} \big).
\end{align*}
Application of the estimate \eqref{e:est-j} gives
\begin{align*}
&\|\mathcal{P}(t_0,\cdot)\|_{H^{s+a}} \\
&\le \| v(t_0)-e^{i(t_0 -\delta (m-1))\partial_{xx}}v_{m-1} \|_{H^{s+a}}
	+\sum_{j=1}^{m-1} \| v_{j}-e^{i \delta\partial_{xx}}v_{j-1} \|_{H^{s+a}} \\
&\le Cme^{t_0}\|V\|_{H^r}\|f\|_{H^s}.
\end{align*}
Therefore, we obtain the desired estimate \eqref{want}.

\end{document}